\documentclass[11pt]{article}
\usepackage[T1]{fontenc}
\usepackage{lmodern}
\usepackage{amsmath,amssymb,amsthm}
\usepackage[margin=1in]{geometry}
\usepackage{microtype}
\usepackage{enumitem}
\usepackage{hyperref}
\hypersetup{colorlinks=true,linkcolor=blue,citecolor=blue,urlcolor=blue}
\newcommand{\Q}{\mathbb Q}
\newcommand{\Z}{\mathbb Z}
\newcommand{\Pone}{\mathbb P^1}

\newcommand{\height}{h}
\newtheorem{theorem}{Theorem}[section]
\newtheorem{corollary}[theorem]{Corollary}
\newtheorem{conjecture}[theorem]{Conjecture}
\theoremstyle{remark}

\theoremstyle{plain}
\title{A Counterexample to the Eventual Stability Conjecture}
\author{THOMAS J. TUCKER
\and ZHENG ZHU}
\date{\today}
\begin{document}
\maketitle
\begin{abstract}
We give a counterexample to the eventual stability conjecture over number fields and classify for power maps. 
\end{abstract}

\section{A counterexample to the number field case}

Let $K$ be a number field and let $\phi\in K(z)$ have degree at least two.
Write $\phi^n=f_n/g_n$ with coprime polynomials. For
$\alpha\in K$, eventual stability means that the number of irreducible factors
of $f_n-\alpha g_n$ in $K[z]$, counted with multiplicity, is bounded independently
of $n$. If $\alpha=\infty$, use $g_n$ instead.

A point $\alpha\in\Pone(K)$ is \emph{nonperiodic} if
$\phi^r(\alpha)\ne\alpha$ for every $r\geq1$.

\begin{conjecture}[Jones--Levy {\cite[Conjecture~1.2(1)]{JL}}]
\label{conj:JL}
Let $K$ be a number field and let $\phi\in K(z)$ have degree $d\geq2$.
If $\alpha\in\Pone(K)$ is nonperiodic under $\phi$, then
$(\phi,\alpha)$ is eventually stable over $K$.
\end{conjecture}

The following example disproves Conjecture~\ref{conj:JL}.

\begin{theorem}\label{thm:counterexample}
Let $K=\Q$, $\phi(z)=z^6$, and $\alpha=-1$. Then $\alpha$ is nonperiodic, but
$(\phi,\alpha)$ is not eventually stable. More precisely,
\begin{equation}\label{eq:counterexample}
 z^{6^n}+1=\prod_{j=0}^{n}\Phi_{2^{n+1}3^j}(z)
 \qquad(n\geq1)
\end{equation}
is its complete irreducible factorization over $\Q$.
\end{theorem}

\begin{proof}
The orbit is $-1\mapsto1\mapsto1$, so no positive iterate returns to $-1$.
The depth-$n$ preimage polynomial is $z^{6^n}+1$.
For every positive integer $N$,
\[
 z^N+1=\frac{z^{2N}-1}{z^N-1}
       =\prod_{\substack{r\mid 2N\\r\nmid N}}\Phi_r(z).
\]
When $N=6^n=2^n3^n$, the divisors of $2N$ that do not divide $N$ are exactly
$2^{n+1}3^j$, $0\leq j\leq n$. This proves \eqref{eq:counterexample}.
The cyclotomic polynomials in that formula are irreducible over $\Q$ and
distinct. Thus the factor count is exactly $n+1$.

All factors have multiplicity one: the derivative is $6^nz^{6^n-1}$ and the
constant term is $1$. The count tends to infinity.
\end{proof}

Conjugating by $\mu(z)=z+1$ gives the same phenomenon with $\alpha=0$:
\[
 f(z)=(z-1)^6+1,\qquad 0\mapsto2\mapsto2,
 \qquad f^n(z)=(z-1)^{6^n}+1.
\]
Translation preserves irreducibility, so $f^n$ also has exactly $n+1$
irreducible factors over $\Q$.

\section{The classification for power maps}

\begin{theorem}\label{thm:classification}
Let $K$ be a number field, let $d\geq2$, and let $\alpha\in K$ be finite. Set
\[
 a_n(K,d,\alpha)
   =\#\{\text{irreducible factors of }z^{d^n}-\alpha
                  \text{ in }K[z]\}\qquad(n\geq1),
\]
with multiplicity. For each prime $\ell$ and positive integer $N$, define
\[
 v_\ell(N):=\max\{e\in\Z_{\geq0}:\ell^e\mid N\}.
\]

\begin{enumerate}[label=\textup{(\roman*)},leftmargin=*]
\item If $\alpha=0$, then $a_n(K,d,\alpha)=d^n$.
\item If $\alpha\ne0$ is not a root of unity, the numbers
$a_n(K,d,\alpha)$ are bounded independently of $n$.
\item If $\alpha$ has multiplicative order $m$, let $F=\Q(\alpha)$ and
\[
 T_n=\prod_{\substack{\ell\mid d\\\ell\nmid m}}
                       \bigl(nv_\ell(d)+1\bigr).
\]
Then
\begin{equation}\label{eq:exact}
 a_n(F,d,\alpha)=T_n,\qquad
 T_n\leq a_n(K,d,\alpha)\leq[K:F]T_n.
\end{equation}
In particular, $(z^d,\alpha)$ is eventually stable if and only if every prime
divisor of $d$ divides $m$.
\end{enumerate}
If $\alpha$ has multiplicative order $m$, then $\alpha$ is periodic under
$z^d$ if and only if $\gcd(m,d)=1$.
\end{theorem}

\begin{proof}
Part~(i) is immediate. For part~(ii), let $\height$ denote the absolute
logarithmic Weil height. Northcott's theorem and Kronecker's theorem give
a constant $\epsilon_K>0$ such that
\begin{equation}\label{eq:gap}
 \height(b)\geq\epsilon_K
 \qquad(b\in K^*,\ b\text{ not a root of unity}),
\end{equation}
see \cite[Theorems~3.7 and~3.8]{Silverman}.

Suppose that $\alpha\ne0$ is not a root of unity. For any $D\geq1$, let $H$
be a monic irreducible factor of $z^D-\alpha$, of degree $r$, and set
$b=(-1)^rH(0)\in K^*$. Multiplying the $D$-th powers of the roots of $H$ gives
\begin{equation}\label{eq:constant}
 b^D=\alpha^r.
\end{equation}
Thus $b$ is not a root of unity, and
\[
 r\height(\alpha)=D\height(b)\geq D\epsilon_K.
\]
Every irreducible factor has degree at least $D\epsilon_K/\height(\alpha)$.
Their degrees sum to $D$, so their number is at most
$\height(\alpha)/\epsilon_K$, independently of $D$.
Taking $D=d^n$ proves part~(ii).

For part~(iii), suppose that $\alpha$ has order $m$, and let $F=\Q(\alpha)$
and $D=d^n$. Every root of $z^D-\alpha$ is a root of unity. An integer $r\geq1$
occurs as the order of a root if and only if
\begin{equation}\label{eq:order}
 \frac{r}{\gcd(r,D)}=m.
\end{equation}
Roots of $z^D-\alpha$ of the same order are conjugate over $F$, so the
irreducible factors over $F$ correspond bijectively to the orders satisfying
\eqref{eq:order}.

For each prime $\ell$, condition \eqref{eq:order} is equivalent to
\[
 \begin{cases}
 v_\ell(r)=v_\ell(m)+v_\ell(D),&\ell\mid m,\\
 0\leq v_\ell(r)\leq v_\ell(D),&\ell\nmid m.
 \end{cases}
\]
The only choices occur at primes dividing $D$ but not $m$. Hence the number
of admissible orders is
\[
 \prod_{\substack{\ell\mid D\\\ell\nmid m}}(v_\ell(D)+1)=T_n.
\]
Since $z^D-\alpha$ is squarefree, this proves $a_n(F,d,\alpha)=T_n$.

Finally, an irreducible polynomial over $F$ of degree $s$ splits into at most
$e=[K:F]$ irreducible factors over $K$: every root $\beta$ satisfies
\[
 [K(\beta):K]=\frac{[K(\beta):F]}{e}\geq\frac{s}{e}.
\]
Extending the field cannot merge distinct factors, so
$T_n\leq a_n(K,d,\alpha)\leq eT_n$, as claimed.
\end{proof}

\begin{corollary}\label{cor:universal}
For a fixed integer $d\geq2$, the following are equivalent:
\begin{enumerate}[label=\textup{(\roman*)},leftmargin=*]
\item For every number field $K$ and every finite nonperiodic $\alpha\in K$,
the pair $(z^d,\alpha)$ is eventually stable over $K$.
\item The integer $d$ is a prime power.
\end{enumerate}
\end{corollary}

\end{document}